\documentclass[12pt]{article}
\usepackage{amssymb,amsmath,amsfonts,amsthm,amsxtra,setspace}

\newcommand{\op}{\operatorname}

\newcommand{\mc}{\mathcal}
\newcommand{\bsigma}{\mathbf{\mathop{\pmb{\sum}}}}

\newcommand{\noopsort}[1]{}
\newtheorem{prob}{Problem}

\newtheorem{theorem}{Theorem}[section]
\newenvironment{ex}{\begin{trivlist} \item[] {\bf Example.}}{\hspace*{0pt}\end{trivlist}}
\newtheorem{cor}[theorem]{Corollary}
\newenvironment{rem}{\begin{trivlist} \item[] {\bf Remark.}}{\hspace*{0pt}\end{trivlist}}

\newsavebox{\Prfref}

\newsavebox{\prfref}

\newtheoremstyle{ref}
{\topsep}	
{\topsep}	
{\it}
{}
{}
{}
{ }
{\thmname{{\bfseries#1}}\thmnumber{ \textbf{#2\thmnote{\rm #3}\textbf .}}}

\theoremstyle{ref}
\newtheorem{lem}[theorem]{Lemma}
\newtheorem{thm}[theorem]{Theorem}
\newtheorem{prop}[theorem]{Proposition}

\newtheoremstyle{nnref}
{\topsep}	
{\topsep}	
{}
{}
{}
{}
{ }
{\thmname{\textbf{#1}\thmnote{\textrm{ #3}}\textbf{.}}}

\theoremstyle{nnref}

\newtheorem*{defn}{Definition}

\begin{document}
\title{Definable versions of Menger's conjecture}
\author{Franklin D. Tall{$^1$}}

\footnotetext[1]{Research supported by NSERC grant A-7354.\vspace*{2pt}}
\date{\today}
\maketitle

\begin{abstract}
	Menger's conjecture that Menger spaces are $\sigma$-compact is false; it is true for analytic subspaces of Polish spaces and undecidable for more complex definable subspaces of Polish spaces. For non-metrizable spaces, analytic Menger spaces are $\sigma$-compact, but Menger continuous images of co-analytic spaces need not be. The general co-analytic case is still open, but many special cases are undecidable, in particular, Menger co-analytic topological groups. We also prove that if there is a Michael space, then productively Lindel\"of \v{C}ech-complete spaces are $\sigma$-compact. We also give numerous characterizations of proper K-Lusin spaces. Our methods include the Axiom of Co-analytic Determinacy, non-metrizable descriptive set theory, and {Arhangel'ski\u\i}'s work on generalized metric spaces.
\end{abstract}

\renewcommand{\thefootnote}{}
\footnote
{\parbox[1.8em]{\linewidth}{$2000$ Math.\ Subj.\ Class.\ Primary 54A35,
54D45, 03E35, 03E75, 54H05; Secondary 03E15, 03E60.}\vspace*{5pt}}
\renewcommand{\thefootnote}{}
\footnote
{\parbox[1.8em]{\linewidth}{Key words and phrases: Co-analytic, Menger, $\sigma$-compact, productively Lindel\"of,
determinacy, Michael space, topological group, K-analytic, absolute Borel, K-Lusin.}}

\section{Menger co-analytic groups}
We shall assume all spaces are completely regular.  

\begin{defn}
A topological space is \textbf{analytic} if it is a continuous image of $\mathbb{P}$, the space of irrationals. A space is \textbf{Lusin} if it is an injective continuous image of $\mathbb{P}$. (This is the terminology of \cite{RJ}. This term is currently used for a different concept.) A space is \textbf{$K$-analytic} if it is a continuous image of a Lindel\"of \v{C}ech-complete space. A space is \textbf{$K$-Lusin} if it is an injective continuous image of a Lindel\"of \v{C}ech-complete space.
\end{defn}

\begin{defn}
A space is \textbf{co-analytic} if $\beta X\setminus X$ is analytic. In general, we call $\beta X\setminus X$ \textbf{the remainder} of $X$. $b X\setminus X$, for any compactification $b X$ of $X$, is called \textbf{a remainder} of $X$.
\end{defn}

\begin{defn}
A space is \textbf{Menger} if whenever $\{\mc{U}_n : n < \omega\}$ is a sequence of open covers, there exist finite $\mc{V}_n$, $n < \omega$, such that $\mc{V}_n \subseteq \mc{U}_n$ and $\bigcup \{\mc{V}_n :  n < \omega\}$ is a cover.
\end{defn}

Arhangel'ski\u\i ~\cite{Ar2} proved that Menger analytic spaces are $\sigma$-compact, generalizing Hurewicz's classic theorem that Menger completely metrizable spaces are $\sigma$-compact. Menger's conjecture was disproved in \cite{MF}, where Miller and Fremlin also showed it undecidable whether Menger co-analytic sets of reals are $\sigma$-compact. In \cite{TT} we proved that Menger \v{C}ech-complete spaces are $\sigma$-compact and obtained various sufficient conditions for Menger co-analytic topological spaces to be $\sigma$-compact.  We continue that study here.  In \cite{TT} we observed that $\mathbf{\Pi}_{1}^1$-determinacy -- which we also call \textbf{CD}: the \emph{Axiom of Co-analytic Determinacy} -- implies Menger co-analytic sets of reals are $\sigma$-compact.  Indeed, \textbf{PD} (\emph{the Axiom of Projective Determinacy}) implies Menger projective sets of reals are $\sigma$-compact \cite{T2}, \cite{TT}.  When one goes beyond co-analytic spaces in an attempt to generalize Arhangel'ski\u\i's theorem, one runs into ZFC counterexamples, but it is not clear whether there is a ZFC co-analytic counterexample.  Assuming $V=L$, there is a counterexample which is a subset of $\mathbb{R}$ \cite{MF}, \cite{TT}.  Here we prove:

\begin{thm}\label{thm1}
\textbf{CD} implies every Menger co-analytic topological group is $\sigma$-compact.
\end{thm}

\begin{rem}
\textbf{CD} follows from the existence of a measurable cardinal \cite{M}.
\end{rem}

We first slightly generalize the \textbf{CD} result quoted above.

\begin{lem}\label{lem1}
\textbf{CD} implies every separable metrizable Menger co-analytic space is $\sigma$-compact.
\end{lem}

In order to prove this, we need some general facts about analytic spaces and perfect maps.

\begin{lem}\label{lemA}
	Metrizable perfect pre-images of analytic spaces are analytic.
\end{lem}

\begin{proof}
	Rogers and Jayne \cite[5.8.9]{RJ} prove that perfect pre-images of metrizable analytic spaces are $K$-analytic, and that $K$-analytic metrizable spaces are analytic \cite[5.5.1]{RJ}.
\end{proof}

\begin{lem}[{ \cite[3.7.6]{E}}]\label{lemB}
	If $f:X\to Y$ is perfect, then for any $B\subseteq Y$, $f_B:f^{-1}(B)\to B$ is perfect.
\end{lem}

\begin{lem}[{ \cite[5.2.3]{RJ}}]\label{lemC}
	If $f$ is a continuous map of a compact Hausdorff $X$ onto a Hausdorff space $Y$ and the restriction of $f$ to a dense subspace $E$ of $X$ is perfect, then $f^{-1}\circ f(E)=E$.
\end{lem}

\begin{lem}\label{lemD}
	Metrizable perfect pre-images of co-analytic spaces are co-analytic.
\end{lem}

\begin{proof}
	Let $M$ be a metrizable perfect pre-image of a co-analytic $X$. Let $p$ be the perfect map. Extend $p$ to $P$ mapping $\beta M$ onto $\beta X$. Then by Lemma \ref{lemC}, $P^{-1}\circ P(M)=M$, i.e. $P^{-1}(X)=M$. Then $P(\beta M\setminus M)=\beta X\setminus X$, since $P$ is onto and points in $M$ map into $X$. By Lemma \ref{lemB}, $P|P^{-1}(\beta X\setminus X)$ is perfect. But then $\beta M\setminus M$ is analytic by Lemma \ref{lemA}, so $M$ is co-analytic.
\end{proof}

\begin{proof}[Proof of Lemma \ref{lem1}]
	Let $X$ be separable metrizable Menger co-analytic. It is folklore (see e.g. \cite{K}) that every separable metrizable space $X$ is a perfect image of a $0$-dimensional one, and hence of a subspace $M$ of the Cantor space $\mathbb{K}\subseteq\mathbb{R}$. Then $M$ is Menger co-analytic, so by \textbf{CD} is $\sigma$-compact. But then so is $X$.
\end{proof}

\begin{lem}[{ \cite{Ar3}}]\label{lem2}
A topological group with Lindel\"of remainder is a perfect pre-image of a metrizable space.
\end{lem}

Since analytic spaces are Lindel\"of, a co-analytic group is a perfect pre-image of a metrizable space.  Since Menger spaces are Lindel\"of, a Menger co-analytic topological group $G$ is a perfect pre-image of a separable metrizable space $M$. In \cite{TT}, we proved \textit{perfect images of co-analytic spaces are co-analytic}, so $M$ is co-analytic and Menger and therefore $\sigma$-compact by \textbf{CD} and Lemma \ref{lem1}. Then $G$ is $\sigma$-compact as well.\qed

After hearing about Theorem \ref{thm1}, S. Tokg\"oz \cite{To} proved:

\begin{prop}\label{prop1}
$V=L$ implies there is a Menger co-analytic group which is not $\sigma$-compact.
\end{prop}

\section{Productively Lindel\"of co-analytic spaces}
\begin{defn}
A space $X$ is \textbf{productively Lindel\"of} if for every Lindel\"of space $Y$, $X \times Y$ is Lindel\"of.
\end{defn}

We have extensively studied productively Lindel\"of spaces \cite{AAJT, AT, BT, DTZ, T2, T1, TTs}, as have other authors.  Since productively Lindel\"of spaces consistently are Menger \cite{T1, AAJT, T2, RZ} it is natural to ask:

\begin{prob}\label{prob1}
Are productively Lindel\"of co-analytic spaces $\sigma$-compact?
\end{prob}

\begin{defn}
A \textbf{Michael space} is a Lindel\"of space whose product with the space $\mathbb{P}$ of irrationals is not Lindel\"of.
\end{defn}

It is consistent that there is a Michael space, but it is not known whether there is one from ZFC.  If there is no Michael space, then the space $\mathbb{P}$ of irrationals is productively Lindel\"of, co-analytic, nowhere locally compact, but not $\sigma$-compact.  We shall prove:

\begin{thm}\label{thm4}
CH implies productively Lindel\"of co-analytic spaces which are nowhere locally compact are $\sigma$-compact.
\end{thm}

I do not know whether the unwanted ``nowhere locally compact" clause can be removed.  It assures us that $\beta X \setminus X$ is dense in $\beta X$.  Laying the groundwork for proving Theorem \ref{thm4}, we need some definitions and previous results.

\begin{defn}[\cite{Ar1}]
A space is of \textbf{countable type} if each compact set is included in a compact set of countable character.
\end{defn}

\begin{lem}[{~\cite{HI}}]\label{lem5}
A completely regular space is of countable type if and only if some (all) remainder(s) are Lindel\"of.
\end{lem}

\begin{defn}[{\cite{A}}]
A space is \textbf{Alster} if each cover by $G_\delta$'s such that each compact set is included in the union of finitely many members of the cover has a countable subcover.
\end{defn}

\begin{lem}[{~\cite{AAJT, T2}}]\label{lem6}
Alster spaces of countable type are $\sigma$-compact.
\end{lem}

\begin{lem}[{~\cite{A}}]\label{lem7}
CH implies productively Lindel\"of spaces of weight $\leq \aleph_1$ are Alster.
\end{lem}

We can now prove Theorem \ref{thm4}.  Let $X$ be productively Lindel\"of, co-analytic, and nowhere locally compact.  $\beta X \setminus X$ is analytic and hence Lindel\"of and separable.  It is dense in $\beta X$, so $w(\beta X)$ and hence $w(X) \leq 2^{\aleph_0}=\aleph_1$.  Then $X$ is Alster.  Since $\beta X \setminus X$ is Lindel\"of, $X$ has countable type, so it is $\sigma$-compact.
\hfill\qed

For metrizable spaces, Repov\'s and Zdomskyy \cite{RZ} proved:

\begin{prop}\label{prop8}
If there is a Michael space and \textbf{CD} holds, then every co-analytic productively Lindel\"of metrizable space is $\sigma$-compact.
\end{prop}

We would like to drop the metrizability assumption, using:

\begin{lem}[{~\cite{RZ}}]\label{lem9}
If there is a Michael space, then productively Lindel\"of spaces are Menger.
\end{lem}

As in \cite{TT}, we run up against the unsolved problem:

\begin{prob}\label{prob2}
Is it consistent that co-analytic Menger spaces are $\sigma$-compact?
\end{prob}

However, we can apply the various partial results in the previous section and \cite{TT} to obtain:

\begin{thm}\label{thm10}
Suppose there is a Michael space and \textbf{CD} holds.  Then if $X$ is co-analytic and productively Lindel\"of, then $X$ is $\sigma$-compact if either:
\begin{enumerate}
\item{
closed subspaces of $X$ are $G_\delta$'s,}
\item[]{or}
\item{$X$ is a $\mathbf{\Sigma}$-space,}
\item[]{or}
\item{$X$ is a $p$-space,}
\item[]{or}
\item{$X$ is a topological group.}
\end{enumerate}
\end{thm}

\begin{proof}
These conditions all imply under \textbf{CD} that Menger co-analytic spaces are $\sigma$-compact. $\bsigma$-spaces and p-spaces are discussed in Section 3.
\end{proof}

The two hypotheses of Theorem \ref{thm10} are compatible, since it is well-known that CH is compatible with the existence of a measurable cardinal, and that CH implies the existence of a Michael space \cite{Mi}. Various other hypotheses about cardinal invariants of the continuum also imply the existence of a Michael space -- see e.g. \cite{Moore2}. These are all compatible with \textbf{CD}.

We also have:

\begin{thm}\label{thm12}
There is a Michael space if and only if productively Lindel\"of \v{C}ech-complete spaces are $\sigma$-compact.
\end{thm}
\begin{proof}
If there is no Michael space, the space of irrationals is productively Lindel\"of, and of course it is \v{C}ech-complete but not $\sigma$-compact.  If there is a Michael space, productively Lindel\"of spaces are Menger, but we showed in \cite{TT} that Menger \v{C}ech-complete spaces are $\sigma$-compact.
\end{proof}

Repov\'s and Zdomskyy \cite{RZ} prove:

\begin{prop}\label{prop13}
Suppose $\op{cov}(\mc{M}) > \omega_1$, and there is a Michael space.  Then every productively Lindel\"of $\Sigma_{2}^1$ subset of the Cantor space is $\sigma$-compact.
\end{prop}

\begin{ex}
The metrizability condition cannot be removed; \emph{Okunev's space} is a productively Lindel\"of continuous image of a co-analytic space, but is not $\sigma$-compact (see \cite{BT}).  In more detail, consider the Alexandrov duplicate $A$ of the space $\mathbb{P}$ of irrationals. $A$ is co-analytic, since it has a countable remainder with a countable base. A countable metrizable space is homeomorphic to an $F_\sigma$ in the Cantor space, and so is analytic. Okunev's space is obtained by collapsing the non-discrete copy of $\mathbb{P}$ in $A$ to a point.  Note that Okunev's space is not co-analytic. To see this, if it were, it would be of countable type by Lemma \ref{lem5}. In \cite{BT} we showed that this space is Alster but not $\sigma$-compact, which would contradict Lemma \ref{lem6}.
\end{ex}

\section{K-analytic and K-Lusin spaces}
We take the opportunity to make some observations about \textit{K-analytic}, \textit{K-Lusin}, \textit{absolute Borel}, \emph{Frol\'ik}, and what {Arhangel'ski\u\i} \cite{Ar4} calls \emph{Borelian of the first type} spaces. These are all attempts to generalize concepts of Descriptive Set Theory beyond separable metrizable spaces.

\begin{defn}[{ \cite{BT}}]
A space is \textbf{Frol\'ik} if it is homeomorphic to a closed subspace of a countable product of $\sigma$-compact spaces.
\end{defn}

\begin{defn}
	A space $X$ is \textbf{absolute Borel} if it is in the $\sigma$-algebra generated by the closed sets of $\beta X$. A space $X$ is \textbf{Borelian of the first type} if it is in the $\sigma$-algebra generated by the open sets of $\beta X$.
\end{defn}

\begin{defn}
	A space is \textbf{projectively $\sigma$-compact} (\textbf{projectively countable}) if its continuous images in separable metrizable spaces are all $\sigma$-compact (countable).
\end{defn}

Frol\'ik \cite{F} showed that each Frol\'ik space is absolute $K_{\sigma\delta}$ (and therefore Lindel\"of), i.e. the intersection of countably many $\sigma$-compact subspaces of its \v{C}ech-Stone compactification (and conversely), and also is the continuous image of a \v{C}ech-complete Frol\'ik space, so that Frol\'ik spaces are absolute Borel and {$K$-analytic}. $K$-Lusin spaces are clearly $K$-analytic; $K$-Lusin spaces are also Frol\'ik \cite[5.8.6]{RJ}. Since $K$-analytic 	metrizable spaces are analytic and analytic Menger spaces are $\sigma$-compact \cite{Ar2}, we see that Menger K-analytic spaces are projectively $\sigma$-compact \cite{BT}.  In \cite{T1} we proved that projectively $\sigma$-compact Lindel\"of spaces are \emph{Hurewicz}, so we conclude:

\begin{thm}
Menger K-analytic spaces are Hurewicz.
\end{thm}

\emph{Hurewicz} is a property strictly between $\sigma$-compact and Menger.  A space is \textbf{Hurewicz} if every \v{C}ech-complete space including it includes a $\sigma$-compact subspace including it (This is equivalent to the usual definition -- see \cite{T2}). This theorem may give some inkling as to why it seems to be hard to find topological properties that imply Hurewicz spaces are $\sigma$-compact which don't in fact imply Menger spaces are $\sigma$-compact. There are, however, Hurewicz subsets of $\mathbb{R}$ which are not $\sigma$-compact --- see e.g. \cite{Ts}.

There is a projectively $\sigma$-compact Frol\'ik space which is not $\sigma$-compact (Okunev's space -- see \cite{BT}). Okunev's space is also not \v{C}ech-complete, since Menger \v{C}ech-complete spaces are $\sigma$-compact \cite{TT}.  There is a Frol\'ik subspace of $\mathbb{R}$ which is not \v{C}ech-complete, since ``\v{C}ech-complete" translates into being a $G_{\delta}$, and we know the Borel hierarchy is non-trivial. There are of course analytic subsets of $\mathbb{R}$ which are not absolute Borel and hence not Frol\'ik. Moore's L-space \cite{Moore} is projectively countable but not K-analytic. The reason is that all its points are $G_\delta$'s, which contradicts projectively countable for K-analytic spaces \cite[5.4.3]{RJ}.

Since K-Lusin spaces are Frol\'ik, it is worth mentioning that:

\begin{prop}[{ \cite{BT}}]
	There are no Michael spaces if and only if every Frol\'ik space is productively Lindel\"of.
\end{prop}

	We could add to this ``if and only if every K-Lusin space is productively Lindel\"of''.
\begin{proof}
	$\mathbb{P}$ is K-Lusin.
\end{proof}

Also of interest is:

\begin{prop}[ {\cite[2.5.5]{RJ}}]
	K-analytic spaces are \textit{powerfully Lindel\"of}, i.e. their countable powers are Lindel\"of -- in fact they are K-analytic.
\end{prop}

\begin{thm}\label{thm32}
	Co-analytic Menger K-analytic spaces are $\sigma$-compact.
\end{thm}

\begin{cor}
	Suppose there is a Michael space. Then co-analytic productively Lindel\"of K-analytic spaces are $\sigma$-compact.
\end{cor}

Compare with \ref{prop8}.

The Corollary follows from \ref{lem9}. In order to prove \ref{thm32} we need to know:

\begin{defn}[{\cite{Ar4}}]
	A completely regular space is called an \textbf{s-space} if there exists a \textbf{countable open source} for $X$ in some compactification $bX$ of $X$, i.e. a countable collection $\mathcal{S}$ of open subsets of $bX$ such that $X$ is a union of some family of intersections of non-empty subfamilies of $\mathcal{S}$.
\end{defn}

We also need to know about \textbf{p-spaces} and \textbf{$\bsigma$-spaces}, but do not need their internal characterizations. What we need are:

\begin{lem}[{ \cite{Ar1}}]
	A completely regular space is \emph{\textbf{Lindel\"of {p}}} if it is the perfect pre-image of a separable metrizable space.
\end{lem}

\begin{lem}[{ \cite{Naga}}]
	A completely regular space is \textbf{\emph{Lindel\"of $\bsigma$}} if and only if it is the continuous image of a Lindel\"of p-space.
\end{lem}

\begin{lem}
	An analytic space has a countable network and hence (see e.g \cite{G}) is Lindel\"of and a $\bsigma$-space.
\end{lem}

\begin{lem}[{ \cite{Ar4}}]\label{lem37}
	$X$ is a Lindel\"of p-space if and only if it is a Lindel\"of $\bsigma$-space and an s-space.
\end{lem}

\begin{lem}[{ \cite{Ar4}}]
	$X$ is Lindel\"of $\bsigma$ if and only if its remainder is an s-space.
\end{lem}

\begin{lem}[{ \cite{Ar3}}]
	$X$ is a Lindel\"of p-space if and only if its remainder is.
\end{lem}

\begin{proof}[Proof of Theorem \ref{thm32}]
	Such a space $X$ is a Lindel\"of p-space, since both it and its remainder are Lindel\"of $\bsigma$. Let $X$ map perfectly onto a metrizable $M$. Then $M$ is analytic and Menger, so is $\sigma$-compact, so $X$ is also.
\end{proof}

\begin{thm}\label{thm310}
	Co-analytic Menger absolute Borel spaces are $\sigma$-compact.
\end{thm}

To see this, we introduce:

\begin{defn}
	Given a family of sets $\mathcal{S}$, Rogers and Jayne \cite{RJ} say that a set is a \textbf{Souslin $\mathcal{S}$-set} if it has a representation in the form
	\begin{equation*}
		\bigcup_{\sigma\in{^\omega\omega}}\bigcap_n\textbf{S}(\sigma|n)
	\end{equation*}
	with \textbf{S}$(\sigma|n)\in\mathcal{S}$ for all finite sequences of positive integers.
\end{defn}

Rogers and Jayne prove:
\begin{lem}[{ \cite[2.5.4]{RJ}}]\label{lem314}
	The family $\mathcal{A}$ of $K$-analytic subsets of a completely regular space is closed under the Souslin operation i.e. every Souslin $\mathcal{A}$-set is in $\mathcal{A}$; if a family is closed under the Souslin operation, it is closed under countable intersections and countable unions.
\end{lem}

\begin{cor}\label{cor313}
	Absolute Borel spaces are $K$-analytic.
\end{cor}

\begin{proof}
	This is well-known. In $\beta X$, closed subsets are compact; compact spaces are $K$-analytic.
\end{proof}

Theorem \ref{thm310} now follows from \ref{thm32}.

\begin{thm}\label{thm314}
	Every Lindel\"of Borelian space of the first type is K-analytic.
\end{thm}

\begin{proof}
	We proceed by induction on subspaces of a fixed compact space. For the basis step, note that open subspaces of a compact space are locally compact, while Lindel\"of locally compact spaces are $\sigma$-compact. For the successor stage, assume a Lindel\"of Borelian set of the first type is the union (intersection) of countably many K-analytic subspaces. By Lemma \ref{lem314}, the union (intersection) is K-analytic and hence Lindel\"of. The limit stage is trivial.
\end{proof}

{Arhangel'ski\u\i} \cite{Ar4} proved that Borelian sets of the first type are s-spaces. This is interesting because:

\begin{thm}\label{lem39}
	Every absolute Borel s-space is a Lindel\"of p-space.
\end{thm}

\begin{proof}
	We induct on Borel order. The basis step is trivial. We need to show s-spaces which are the countable union (intersection) of Lindel\"of p-spaces are Lindel\"of \emph{p}. By \ref{lem37} it suffices to show they are Lindel\"of $\bsigma$. Let $\{X_n\}_{n<\omega}$ be Lindel\"of \textit{p}. Let $\sum_{n<\omega} X_n$ be the disjoint sum of the $X_n$'s. Then $\sum_{n<\omega} X_n$ is clearly Lindel\"of \textit{p}. Consider the natural map $\sigma$ from $\sum_{n<\omega} X_n$ to $\bigcup_{n<\omega} X_n$ obtained by identifying all copies of a point $x\in\bigcup_{n<\omega}X_n$ which are in $\sum_{n<\omega} X_n$. $\sigma$ is continuous, so $\bigcup_{n<\omega} X_n$ is Lindel\"of $\bsigma$.
	
	Now consider $\prod_{n<\omega} X_n$. This is also Lindel\"of \textit{p} \cite{Ar1} and so then is the diagonal $\Delta$. Define $\pi\left(\langle x,x, \dots\rangle\right)=x$. Then $\pi$ is continuous and maps $\Delta$ onto $\bigcap_{n<\omega} X_n$, which is therefore Lindel\"of $\bsigma$.
\end{proof}

Note Okunev's space is Lindel\"of absolute $F_{\sigma\delta}$ but is not s, since it is not of countable type \cite{BT}, while s-spaces are \cite{Ar4}. By Theorem \ref{thm32}, Okunev's space is not co-analytic.

Borel sets of reals are of course analytic; Okunev's space shows that Lindel\"of absolute Borel spaces need not be analytic, since it is Menger but not $\sigma$-compact. Compact spaces are Borelian of the first type, so the latter spaces need not be analytic.

A somewhat smaller class of spaces than the $K$-analytic ($K$-Lusin) ones is comprised of what Rogers and Jayne call the \textit{proper $K$-analytic} (\textit{proper $K$-Lusin}) spaces.

\begin{defn}
	A space is \textbf{proper $K$-analytic} if it is the perfect pre-image of an analytic subspace of $\mathbb{R}^\omega$. A space is \textbf{proper K-Lusin} if it is the perfect pre-image of a Lusin subspace of $\mathbb{R}^\omega$.
\end{defn}

Rogers and Jayne \cite{RJ} prove that a space is proper K-Lusin if and only if both it and its remainder are K-analytic. It follows that a space is proper K-Lusin if and only if it and its remainder are K-Lusin. They also prove that K-Lusin spaces are absolute $K_{\sigma\delta}$, i.e. what we have called Frol\'ik. It follows that proper K-Lusin spaces are both $K_{\sigma\delta}$ and $G_{\delta\sigma}$, i.e. countable unions of \v{C}ech-complete spaces. We shall provide a large number of equivalences for ``proper K-Lusin'' below.

Proper K-analytic spaces are p-spaces, and their continuous real-valued images are analytic, so:

\begin{thm}
	Menger proper K-analytic spaces are $\sigma$-compact.
\end{thm}

\begin{cor}
	Menger proper K-Lusin spaces are $\sigma$-compact.
\end{cor}

\begin{lem}[{ \cite{RJ}}]\label{lem319}
	Let $\mathfrak{Z}(Y)$ be the collection of zero-sets of $Y$. Then $X$ is proper $K$-analytic if and only if $X\in\textbf{S}(\mathfrak{Z}(\beta X))$.
\end{lem}

\begin{thm}
	A space is proper K-analytic if and only if it is a K-analytic p-space.
\end{thm}

\begin{proof}
	By definition, a proper K-analytic space is a p-space. By \ref{lem319} and \ref{lem314} it is K-analytic. Conversely, if $X$ is a K-analytic p-space, it maps perfectly onto a separable metrizable analytic space, which embeds into $\mathbb{R}^\omega$.
\end{proof}

Note that zero-sets are closed $G_\delta$'s, so that the \textit{absolute Baire sets}, i.e. the elements of the $\sigma$-algebra generated by the zero-sets, are both Lindel\"of Borelian of the first type and absolute Borel. 

\begin{cor}
	Menger absolute Baire spaces are $\sigma$-compact.
\end{cor}

Mixing Rogers and Jayne with {Arhangel'ski\u\i}, we have:

\begin{thm}\label{thm318}
	The following are equivalent:
	\begin{itemize}
		\item[(a)] $X$ is proper K-Lusin,
		\item[(b)] $X$ and its remainder are K-Lusin,
		\item[(c)] $X$ and its remainder are both Frol\'ik,
		\item[(d)] $X$ is Lindel\"of Borelian of the first type,
		\item[(e)] $X$ is absolute Borel and Lindel\"of p,
		\item[(f)] $X$ is absolute Borel and of countable type.
	\end{itemize}
\end{thm}

\begin{proof}
	We have already proved that (a), (b) and (c) are equivalent. (c) implies (d), since $X$ is Lindel\"of absolute $G_{\delta\sigma}$. If $X$ is Lindel\"of Borelian of the first type, it is K-analytic, but so is its remainder, so (d) implies (b). If $X$ is absolute Borel, it is K-analytic and its remainder is Borelian of the first type. If $X$ is Lindel\"of p, so is its remainder, so (e) implies (b). (b) implies a proper K-Lusin space and its remainder are both K-analytic spaces, hence Lindel\"of $\bsigma$ spaces, so they are p-spaces. Thus (b) implies (e). (e) implies (f) since p-spaces are of countable type \cite{Ar1}. (f) implies the remainder of $X$ is Lindel\"of Borelian of the first type, and so is K-analytic. Then since $X$ is K-analytic, (f) implies (b).
\end{proof}

We know that Menger proper K-analytic (a fortiori, proper K-Lusin) spaces are $\sigma$-compact, but Menger K-analytic spaces may not be.

\begin{prob}
	Are Menger K-Lusin spaces $\sigma$-compact?
\end{prob}

An interesting fact about K-Lusin spaces is that:

\begin{lem}[{ \cite[5.4.3]{RJ}}]\label{lem322}
	The following are equivalent for a K-Lusin $X$:
	\begin{itemize}
		\item[(a)] $X$ includes a compact perfect set;
		\item[(b)] $X$ admits a continuous real-valued function with uncountable range;
		\item[(c)] $X$ is not the countable union of compact subspaces which include no perfect subsets. In particular, if $X$ is not $\sigma$-compact, it includes a compact perfect set.
	\end{itemize}
\end{lem}

From this, we can conclude that Okunev's space is not K-Lusin, since it is not $\sigma$-compact but doesn't include a compact perfect set.

Indeed we have:

\begin{defn}
	A space is \textbf{Rothberger} if whenever $\{\mathcal{U}_n\}_{n<\omega}$ are open covers, there exists a cover $\{{U}_n\}_{n<\omega},U_n\in\mathcal{U}_n$.
\end{defn}

Thus \textit{Rothberger} is a strengthening of \textit{Menger}.

\begin{lem}[{ \cite{AA}}]
	Rothberger spaces do not include a compact perfect set.
\end{lem}

\begin{thm}
	K-analytic Rothberger spaces are projectively countable.
\end{thm}

\begin{proof}
	They are projectively $\sigma$-compact.
\end{proof}

\begin{cor}
	K-Lusin Rothberger spaces are $\sigma$-compact.
\end{cor}

\begin{proof}
	This follows from \ref{lem322}.
\end{proof}

\begin{rem}
	Projectively countable Lindel\"of spaces are always Rothberger \cite{T1}; thus Okunev's space is Rothberger \cite{BT}. The assertion that Rothberger spaces are projectively countable is equivalent to \textit{Borel's Conjecture} \cite{T1}.
\end{rem}

Here are some more problems we have not been able to solve:

\begin{prob}
Does \textbf{CD} imply co-analytic Hurewicz spaces are $\sigma$-compact?
\end{prob}

\begin{prob}
Are Lindel\"of co-analytic projectively $\sigma$-compact spaces $\sigma$-compact?
\end{prob}

Note $V=L$ implies there is a co-analytic Hurewicz group of reals that is not $\sigma$-compact \cite{To}.

\nocite{*}
\bibliographystyle{acm}
\bibliography{MengerGroups.bib}

{\rm Franklin D. Tall, Department of Mathematics, University of Toronto, Toronto, Ontario M5S 2E4, CANADA}

{\it e-mail address:} {\rm f.tall@math.utoronto.ca}

\end{document}